\documentclass[a4paper,12pt]{article}
\usepackage[margin=1in]{geometry}
\newcommand{\sign}{\operatorname{sgn}}
\usepackage{times, url, booktabs}
\usepackage{amsmath,amssymb,amsthm,amsfonts,graphicx,float,mathtools}
\usepackage{colortbl,xcolor,tikz,array,comment,breqn,bm}
\usetikzlibrary{arrows.meta}

\newtheorem{theorem}{Theorem}[section]

\newtheorem{lemma}[theorem]{Lemma}

\theoremstyle{definition}
\newtheorem{definition}[theorem]{Definition}
\newtheorem{eg}[theorem]{Example}
\newtheorem{rem}[theorem]{Remark}

\numberwithin{equation}{section}
\usetikzlibrary{trees}
\usepackage{listofitems}
\DeclarePairedDelimiter{\ceil}{\lceil}{\rceil}

\usepackage{hyperref}
\begin{document}
\setcounter{page}{1}

\begin{center}
{\LARGE \bf Circular Superpatterns and Zigzag constructions}
\vspace{8mm}

{\large\bf M. Hariprasad\,$^{1}$, Raisa D'Souza\,$^{2}$}

\vspace{3mm}

$^{1}$School of Mathematics and Natural Sciences,\\
Chanakya University, Bengaluru 562110, India\\
E-mail: \url{mhariprasadkansur@gmail.com}

\vspace{2mm}

$^{2}$School of Arts and Sciences,\\
Azim Premji University, Bengaluru, India \\
E-mail: \url{raisadsouza1989@gmail.com}

\usetikzlibrary{positioning}

\end{center}
\vspace{10mm}

\noindent
{\bf Abstract:} In this article, we introduce the notion of circular $k$-superpatterns, permutations that contain all $k$ length patterns up to rotation equivalence. We present a construction of circular superpattern using the linear $(k-1)$ superpattern, explicitly giving an upper bound on its length. Motivated by the zigzag framework of Engen and Vatter, we adapt and simplify their score function to the circular setting and analyze its parity properties. For odd $k$ we construct a circular $k$-superpattern using zigzag words and prove its correctness.
\vspace{10mm}
\section{Introduction}

The question ``What is the smallest length of a permutation that contains all permutations of length $k$?" has been asked in different contexts: by Knuth et al in 1972 \cite{chvatal1972selected}, by Chung et al in 1992 \cite{chung1992universal}, by Ashlock and Tillotson in 1993 \cite{ashlock1993construction} and by Arratia in 1999 \cite{arratia1999stanley}. Arratia established an equivalence between two versions of the Stanley and Wilf conjecture on pattern avoidance. As a closely related but complementary problem, he also studied pattern containment and constructed a permutation of length $k^2$ that contains all $k$-patterns, establishing an upper bound for the length, $L_k$, of such permutations. He also conjectured that asymptotically, $L_k\sim \left(\dfrac{k}{e}\right)^2$. Patterns that contain every length $k$ permutation are called $k$-super patterns. The problem of constructing smaller superpatterns has attracted attention as a fundamental question in pattern containment.  

In 2007, Eriksson et al \cite{eriksson2007dense}, showed that $\left(\dfrac{k}{e}\right)^2\le L_k\leq \left[\dfrac{2}{3}+O(1)\right]k^2$ and further conjectured that $L_k\sim \dfrac{k^2}{2}$. In 2009, Miller \cite{miller2009asymptotic} gave a simple construction that showed that $L_k\leq \dfrac{k(k+1)}{2}$, lending some evidence to the conjecture by Errikson et al.

In 2021, Engen and Vatter~\cite{engen2020containing} proposed a construction of length $\ceil*{\dfrac{k^2+1}{2}}$ using an  \emph{infinite zigzag} word. They proved that every $k$-pattern is either a \emph{layered permutation} or a \emph{distant inverse permutation}, and gave embeddings of both types into the zigzag pattern.  Their work revealed striking structural regularities in the space of permutations. Superpatterns containing layered permutations  were also studied by Gray \cite{gray2015bounds}. An improved lower bound for $L_k$ was given by Chroman et al \cite{chroman2021lower} and they extended their approach to other universal type problems.

In this article, we introduce the notion of a \emph{circular $k$-superpattern}, which is effectively, a permutation written on a circle, where patterns are considered equivalent up to rotation.  We present a simple construction that converts a linear $(k-1)$-superpattern into a circular $k$-superpattern, yielding an explicit upper bound on its length.  We also simplify the score function of Engen and Vatter and extend it to the circular setting.  Finally, we present a \emph{zigzag parity construction} that generates a circular superpatterns for odd $k$.

The organization of the paper is as follows. Section~2 reviews background and notation. Section~3 presents the construction of a circular $k$-superpattern from a linear superpattern. Section~4 introduces the zigzag construction and parity-based score function. Section~5 gives the construction and proves its correctness of an odd $k$ circular superpattern.

\section{Background and Notations}

Let $[n] = \{1,2,\dots,n\}$, and let $\mathbf{S}_n$ denote the set of all permutations on $[n]$. In this article, we use  one-line notation for permutations; that is, a permutation $\pi$ of $\{1,2,\dots,n\}$ is written as 
$\pi = (\pi_1, \pi_2, \dots, \pi_n)$.  
Here $\pi_j$ denotes the $j$-th element of $\pi$. To improve readability and avoid nested subscripts, when the index itself has a subscript (for example, $j_k$), 
we use parentheses and write $\pi(j_k)$ instead of $\pi_{j_k}$. 
This notation is used consistently throughout the paper. \par
For a subset $J = \{i_1 < i_2 < \cdots < i_k\} \subseteq [n]$, and a permutation $\pi \in \mathbf{S}_n$ define the restriction
\[
\pi|_J = (\pi(i_1), \pi(i_2), \dots, \pi(i_k)).
\]
A permutation $\pi \in \mathbf{S}_n$ \emph{contains} a pattern $\sigma \in \mathbf{S}_k$ as an order-isomorphic subsequence if for some $J$ we have $\pi(i_p) < \pi(i_q)$ exactly when $\sigma(p) < \sigma(q)$.

\begin{eg}
Let $\pi = (3\,2\,1\,4\,5)$. Then $\pi$ contains $\sigma = (2\,1\,3\,4)$ as an order-isomorphic subsequence for the index set $J = \{2,3,4,5\}$, since $\pi|_J =( 2\,1\,4\,5)$.
\end{eg}

A permutation $\pi \in \mathbf{S}_n$ is a \emph{$k$-superpattern} if it contains every permutation in $\mathbf{S}_k$ as an order-isomorphic subsequence.  The minimal possible such length is denoted by $L(k)$.  
Arratia showed that $L(k) \le k^2$, and Engen and Vatter improved this to 
\[
L(k) \le \ceil*{\frac{k^2+1}{2}}.
\]

\paragraph{Circular superpatterns.}
We extend this notion to the circular setting by treating permutations as numbers arranged around a cycle.  
Two linear permutations are \emph{cyclically equivalent} if one can be obtained from the other by a rotation.  
For example, $(1\,2\,3\,4)$ and $(4\,1\,2\,3)$ are cyclically equivalent. 

\begin{definition}[Rotation of a permutation]
Let $\sigma = (\sigma_1, \sigma_2, \ldots, \sigma_k) \in S_k$. 
For an integer $r$ with $0 \le r < k$, the \emph{rotation of $\sigma$ by $r$ positions to the left}, 
denoted by $\rho_r(\sigma)$, is the permutation obtained by cyclically shifting the entries of $\sigma$ to the left:
\[
\rho_r(\sigma) = (\sigma_{r+1}, \sigma_{r+2}, \ldots, \sigma_k, \sigma_1, \ldots, \sigma_r).
\]
Two permutations $\sigma, \tau \in S_k$ are said to be \emph{cyclically equivalent} if 
$\tau = \rho_r(\sigma)$ for some $r$.
\end{definition}

\begin{definition}[Circular containment of a pattern]
A permutation $\pi \in S_n$ \emph{circularly contains a pattern} $\sigma \in S_k$ 
if $\pi$ has a subsequence that is order isomorphic to a permutation that is cyclically equivalent to $\sigma$. That is, there exists a rotation $\rho_r(\sigma)$ of $\sigma$ and indices 
$ i_1 < i_2 < \cdots < i_k $, 
such that the subsequence 
\[
(\pi(i_1), \pi(i_2), \ldots, \pi(i_k))
\]
is order-isomorphic to $\rho_r(\sigma)$. We may also say that \emph{$\pi$ contains $\sigma$ as a circular subsequence.}
\end{definition}

A permutation $\pi \in \mathbf{S}_n$ is a \emph{circular $k$-superpattern} if it contains all permutations in $\mathbf{S}_k$ up to cyclic equivalence.  
Let $L_{\mathrm{circ}}(k)$ denote the minimal possible length of such a permutation.

\paragraph{Zigzag patterns.}
The construction of Engen and Vatter does not begin with a permutation. Instead, it begins with a structured word called a zigzag word. Every symbol appears several times, allowing permutations to be embedded as exact subsequences. A permutation is then obtained from this word by a tie-breaking procedure. Since our circular construction follows the same philosophy, we first introduce zigzag words before explaining how they are converted into permutations.

For completeness we recall the \emph{zigzag pattern} of Engen and Vatter \cite{engen2020containing}.  Fix an integer $q \ge 2$.  
For each integer $j \ge 1$, define the $j$-th \emph{run} $R_j$ as the ordered list of integers from $\{1,2,\dots,q\}$ whose parity matches that of $j$:  
\[
R_j =
\begin{cases}
1,3,5,\dots,q-1,q, & \text{if $j$ is odd},\\[3pt]
q-1,q-3,\dots,4,2, & \text{if $j$ is even}.
\end{cases}
\]
Then the concatenation of the first $m$ runs gives the zigzag pattern
\[
\mathrm{zz}(m,q) = R_1 R_2 \cdots R_m.
\]

We refer to $m$ as the \emph{number of runs} and $q$ as the \emph{alphabet width} of the zigzag word.

\begin{eg}
For $(m,q)=(3,4)$ we have
\[
R_1 = 1,3, \quad
R_2 = 4,2, \quad
R_3 = 1,3,
\]
so that $\mathrm{zz}(3,4) = 1\,3\,4\,2\,1\,3.$
\end{eg} 
Figure \ref{zz(4,4)} and Figure \ref{zz(5,5)} illustrate $zz(4,4)$ and $zz(5,5)$.
Engen and Vatter considered the limiting case $q=\infty$.
\begin{figure}[H]
    \centering
\begin{tikzpicture}
    \def\seq{1,3,4,2,1,3,4,2}

    \foreach [count=\i] \y in \seq {
        \node at (\i,\y) {\y};
    }
\end{tikzpicture}
    \caption{$zz(4,4)$}
    \label{zz(4,4)}
\end{figure}

\begin{figure}[H]
    \centering
\begin{tikzpicture}
    \def\seq{1,3,5,4,2,1,3,5,4,2,1,3,5}

    \foreach [count=\i] \y in \seq {
        \node at (\i,\y) {\y};
    }
\end{tikzpicture}
    \caption{$zz(5,5)$}
    \label{zz(5,5)}
\end{figure}

For a pattern $\sigma = \sigma_1 \sigma_2 \cdots \sigma_k$ we denote the sequence $\sigma^+  = \sigma_1+1 ,\sigma_2+1, \cdots ,\sigma_k+1$.

The upper bounds we prove are summarized in Table \ref{results}.

\begin{table}[h]
\centering
\caption{Upper bounds for $k$-circular superpattern length}\label{results}
\renewcommand{\arraystretch}{1.2}
\begin{tabular}{@{}lcc@{}}
& \textbf{Over alphabet $[k+1]$} 
& \textbf{Over permutations} \\ 
\midrule
\textbf{Length} 
& $\displaystyle \frac{(k-1)(k+1)}{2}$ 
& $\displaystyle \Big\lceil \frac{(k-1)^{2}+1}{2} \Big\rceil + 1$ \\
\bottomrule
\end{tabular}
\end{table}

\section{Circular $k$-superpatterns using linear $k-1$ superpattern}

In this section we construct a circular $k$-superpattern from a linear $(k-1)$-superpattern.  
This yields the general upper bound
\[
L_{\mathrm{circ}}(k) \le L(k-1) + 1.
\]
A numerical example at the end shows that this bound is not always tight.

\begin{theorem}\label{thm:circ_length}
Let $k \ge 2$.  
If $\pi = (\pi_1, \pi_2, \dots, \pi_L)$ is a linear $(k-1)$-superpattern, then the permutation  
\[
\gamma = (\,L+1,\; \pi_1, \pi_2, \dots, \pi_L\,)
\]
is a circular $k$-superpattern, consequently,
\[
L_{\mathrm{circ}}(k) \le L(k-1) + 1.
\]
\end{theorem}

\begin{proof}
We need to show that every cyclic equivalence class of permutations of length $k$ contains a representative that appears as an order-isomorphic subsequence of $\gamma$.
Write
\[
\gamma = (\gamma_0, \gamma_1, \dots, \gamma_L),
\qquad 
\gamma_0 = L+1,\; \gamma_i = \pi_i \ (1 \le i \le L).
\]

Let $\sigma \in \mathbf{S}_k$ be arbitrary.  
Every cyclic equivalence class contains a unique rotation whose first symbol is $k$, so there exists $r$ such that
\[
\rho_r(\sigma) = (\,k,\, \tau_1, \tau_2, \dots, \tau_{k-1}\,),
\]
where $\tau = (\tau_1,\dots,\tau_{k-1}) \in \mathbf{S}_{k-1}$.

Since $\pi$ is a $(k-1)$-superpattern, there exist indices $J'=\{1 \le i_1 < i_2 < \dots < i_{k-1} \le L\}$ such that $\pi|_{J'}=(\pi(i_1), \pi({i_2}), \dots, \pi({i_{k-1}}))$ is order-isomorphic to $\tau$.

Now consider the index set
\[
J = \{\,0,\, i_1,\, i_2, \dots, i_{k-1}\,\}
\subseteq \{0,1,\dots,L\}.
\]
for the permutation $\gamma$. Then,
\[
\gamma|_J
   = (\gamma_0, \gamma({i_1}), \dots, \gamma({i_{k-1}}))
   = (\,L+1,\, \pi({i_1}), \dots, \pi({i_{k-1}})\,).
\]

By construction,
\[
L+1 > \pi_{i_t} \qquad \text{for all } t,
\]
so the first element of $\gamma|_J$ is the unique maximum, just as $k$ is the unique maximum in $\rho_r(\sigma)$.  
The remaining entries of $\gamma|_J$ are order-isomorphic to $\tau$.  
Therefore, $\gamma|_J$ is order isomorphic to $\rho_r(\sigma).$

Since $\sigma$ was arbitrary, $\gamma$ contains a representative of every cyclic equivalence class of $\mathbf{S}_k$ and is therefore a circular $k$-superpattern.

Finally, $\gamma$ has length $L+1$, so
\[
L_{\mathrm{circ}}(k) \le L(k-1) + 1.
\]
\end{proof}

\begin{rem}
    Using the Engen–Vatter bound~\cite{engen2020containing} for $L(k-1)$ gives the explicit inequality,
\[
L_{\mathrm{circ}}(k)
   \;\le\; \left\lceil \frac{(k-1)^2 + 1}{2} \right\rceil + 1
   \;=\; \left\lceil \frac{k^2 - 2k + 4}{2} \right\rceil .
\] 
\end{rem}
\begin{eg}
For $k = 4$, exhaustive search confirms that the minimal circular $4$-superpattern has length~$6$, consistent with Theorem~\ref{thm:circ_length}.  
For $k = 5$, the permutation,
\[
(846271359)
\]
is a circular $5$-superpattern of length~$9$, while the theorem gives $L_{\mathrm{circ}}(5) \le 10$.  
Thus the construction is close to optimal but not tight in this case.
\end{eg}

\section{Exact subsequence containment in Zigzag words}\label{zigzagpattern}
The remainder of the paper studies zigzag constructions. To understand how efficiently a permutation embeds into a zigzag word, we assign a local cost to transitions between consecutive symbols. These local costs combine to form a score function, which determines the minimum number of zigzag runs required for an embedding.
We revisit the \emph{zigzag construction} of Engen and Vatter~\cite[Prop.~10]{engen2020containing}, presenting a streamlined formulation adapted to the circular framework.  
Our focus is the direct embedding of permutation patterns within finite zigzag words, emphasizing exact containment rather than order-isomorphic subsequences.

Let $\sigma \in \mathbf{S}_k$ be arbitrary.  
We examine its placement as a subsequence of $\mathrm{zz}(m,q)$, where $(m,q)$ denote the number of runs and the alphabet width, respectively. \par 
We re-derive proposition 10 of Engen and Vatter that: For every $\sigma \in \mathbf{S}_k$, either $\sigma$ or $\sigma^+$ occurs as an exact subsequence of  $\mathrm{zz}(k,k+1)$  (Illustrated in Figure \ref{exactlyfig} and Figure \ref{liftedfig}). Although Proposition 10 was proved by Engen and Vatter, proving it using our score-function formulation naturally extend it to the circular setting developed in the remainder of the paper.

Further, we prove that for odd $k$, For every $\sigma \in \mathbf{S}_k$, either $\sigma$ or $\sigma^+$ occurs as an exact cyclic subsequence of  $\mathrm{zz}(k-1,k+1)$.
\readlist\seq{1,3,4,2,1,3}
\readlist\target{1,2,3}

\begin{figure}[H]
\centering
\begin{tikzpicture}

\newcount\ptr
\ptr=1

\foreach \i in {1,...,6}{

    \node (n\i) at (\i,\seq[\i]) {\seq[\i]};

    \ifnum\ptr<5
        \ifnum\seq[\i]=\target[\the\ptr]
            \draw[red,thick] (n\i) circle (0.30);
            \global\advance\ptr by1
        \fi
    \fi

}

\end{tikzpicture}
\caption{ The circled entries show an exact embedding of the permutation $\sigma = \{ 1,2,3 \}$  as an exact subsequence of $zz(3,4)$}\label{exactlyfig}
\end{figure}

\readlist\seq{1,3,4,2,1,3}
\readlist\target{4,2,3}

\begin{figure}[H]
\centering
\begin{tikzpicture}

\newcount\ptr
\ptr=1

\foreach \i in {1,...,6}{

    \node (n\i) at (\i,\seq[\i]) {\seq[\i]};

    \ifnum\ptr<5
        \ifnum\seq[\i]=\target[\the\ptr]
            \draw[red,thick] (n\i) circle (0.30);
            \global\advance\ptr by1
        \fi
    \fi

}

\end{tikzpicture}
\caption{ The permutation $\sigma = \{ 3,1,2\}$ does not appear as an exact subsequence of $zz(3,4)$ but lifted permutation $\sigma^+ = \{4,2,3\}$ appears as an exact subsequence (depicted by circled entries).}\label{liftedfig}
\end{figure}

\subsection{Parity and local cost.}\label{sec:cost}
For each $x \in \{1,2,\dots,q\}$ define its \emph{parity sign}
\[
p_x =
\begin{cases}
\phantom{-}1, & \text{if $x$ is even},\\[3pt]
-1, & \text{if $x$ is odd.}
\end{cases}
\]
For $x,y \in \{1,\dots,q\}$ the \emph{local cost} $C_{xy}$ quantifies how many additional
zigzag runs are required to position the number $y$ after $x$. It is defined by
\[
C_{xy}=\begin{cases}
    -1 &\text{no additional run is required to position $y$ after $x$,}\\[3pt]
    \phantom{-}0 &\text{one additional run is required to position $y$ after $x$,}\\[3pt]
    \phantom{-}1 &\text{two additional runs are required to position $y$ after $x$.}
\end{cases}
\]

\begin{lemma}
For $x,y\in\{1,\dots,q\}$, the local cost function is given by
\begin{equation}\label{cxy}
C_{xy}
   = \delta_{xy}
     -\frac{(p_x p_y + 1)}{2}\,
     \operatorname{sgn}(x-y)\,p_x,
\end{equation}
where $\operatorname{sgn}(t)$ denotes the sign of $t$.
\end{lemma}

\begin{proof}
We verify that the right-hand side of \eqref{cxy} reproduces the required values of
$C_{xy}$ in each parity and order case.
\begin{enumerate}
    \item \textbf{Same parity, and $y$ lies in the same run as $x$.}
\smallskip
If $x,y$ are both odd and $y>x$, the zigzag order places $y$ in the same run as $x$.
Then $p_x p_y = 1$, $\operatorname{sgn}(x-y) = -1$, and the right-hand side becomes
\[
\delta_{xy} - \frac{(1+1)}{2}(-1)p_x = 0 - (1)(-1)(-1)= -1,
\]
matching $C_{xy}=-1$.  
The same argument holds when $x,y$ are both even and $y<x$.

\item \textbf{Opposite parity.}
\smallskip
If $x$ and $y$ have different parity, then $p_x p_y = -1$, so
\[
\frac{p_x p_y + 1}{2} = 0,
\]
and the right-hand side reduces to $\delta_{xy}$.  
Since $x\neq y$, this equals $0$, indicating exactly one extra run is needed.  
Thus $C_{xy}=0$, as required.

\item \textbf{Same parity, but $y$ lies two runs after $x$.}

\smallskip
If $x,y$ are both odd and $y<x$, then $y$ appears two runs after $x$.  
Here $p_x p_y = 1$ and $\operatorname{sgn}(x-y)=+1$, giving
\[
0 - \frac{2}{2}(1)p_x = -p_x = 1,
\]
since $p_x=-1$.  
Likewise, when $x,y$ are both even and $y>x$, we obtain $C_{xy}=1$.

\item \textbf{The case $x=y$.}

\smallskip
Since a symbol cannot follow itself within the same zigzag pass, placing $y=x$ again
requires two additional runs.  
In this case $\delta_{xy}=1$ and $\operatorname{sgn}(0)=0$, so $C_{xx} = 1 - 0 = 1$, as required.
\end{enumerate}

\medskip
This covers all parity and ordering possibilities, establishing the formula.
\end{proof}

\begin{definition}[Score function]
    The score of a permutation, $\sigma\in S_k$ is the minimum number of runs in a zigzag word containing $\sigma$ minus the length of $\sigma$.
\end{definition}
\begin{lemma}\label{score}
For a permutation $\sigma = (\sigma_1,\sigma_2,\dots,\sigma_k)$, its \emph{score} is given by the formula,
\begin{equation}\label{ssigma}
S(\sigma) = \sum_{i=1}^{k-1} C_{\sigma_i \sigma_{i+1}}
         + C_{\cdot,\,\sigma_1},    
\end{equation}
where the initial term is
\[
C_{\cdot,\,x} = \frac{1 + p_x}{2}.
\]
\end{lemma}
\begin{proof}
Place the symbols of $\sigma$ one by one, in the given order, using the following greedy rule: when placing the next symbol, place it in the earliest run (counting forward in the zigzag sequence) that preserves the relative order and parity constraints; if no such position in the current or immediately following run is possible, open additional runs as needed.  

By the definition of $C_{\cdot,\sigma_1}$, the greedy procedure requires exactly $C_{\cdot,\sigma_1}$ extra runs before placing $\sigma_1$ (this accounts for whether the first available run can accept $\sigma_1$ or whether an additional run must be inserted first).  After $\sigma_1$ is placed, when we place $\sigma_2$ the number of extra runs that the greedy procedure must open is by definition $C_{\sigma_1\sigma_2}$.  Proceeding inductively, the greedy placement opens exactly $C_{\sigma_i\sigma_{i+1}}$ extra runs when placing $\sigma_{i+1}$ after $\sigma_i$.  Therefore the total number of extra runs opened by this construction equals
\[
C_{\cdot,\sigma_1}+\sum_{i=1}^{k-1} C_{\sigma_i\sigma_{i+1}}.
\]
Since each symbol occupies one position in some run, the total number of runs used by this construction is
\begin{equation}\label{ssigma_proof}
k + \biggl(C_{\cdot,\sigma_1}+\sum_{i=1}^{k-1} C_{\sigma_i\sigma_{i+1}}\biggr).
\end{equation}
The local costs are computed by using the parity and do not depend on choices made elsewhere. Hence, we can say that the minimum number of runs required to space $\sigma$ is (\ref{ssigma_proof}). This establishes, 
equation (\ref{ssigma_proof}).

\end{proof}

\begin{rem}
The optimality of the greedy placement follows by induction. 
Let $\phi$ be any
feasible embedding, and let $g$ denote the greedy embedding. Let $|\phi_{1:j}|$ denote the number of runs required for embedding $\sigma_{1:j}$ via $\phi$.
The base case:  the first symbol of the pattern must be mapped to either $R_1$ or $R_2$ depending on its odd or even.  
The induction hypothesis: Suppose the first
$i-1$ symbols have been placed and $|\phi_{1:i-1}| \geq |g_{1:i-1}|$. 
Implication: 
From the hypothesis, $
g(\sigma_{i-1}) \mapsto (\sigma_{i-1},R_j)$, and $ 
\phi(\sigma_{i-1}) \mapsto (\sigma_{i-1},R_k)$,
then necessarily $k\ge j$. If $k=j$, then $\phi,g : \sigma_{i-1} \to (\sigma_{i-1},R_j)$ (because each symbol in a run appears exactly once). 
Hence the set of feasible placements for $\sigma_i$ is identical in both embeddings, and the greedy algorithm chooses the earliest feasible run leading to $|g_{1:i}| \leq |\phi_{1:i}|$. If $k>j$, then set of feasible placement for $\sigma_i$ in $g$ is more than that of $ \phi$. Hence the $i$-th symbol in the greedy embedding is never placed
later than in $\phi$.

Proceeding inductively, every symbol is assigned to a run whose index is at
most the corresponding run index in $\phi$. Therefore the last run used by the
greedy embedding is no larger than the last run used by $\phi$. Consequently,
the greedy embedding uses the minimum possible number of zigzag runs.

\end{rem}

\begin{rem}
The sign convention follows that of Engen and Vatter, but omits the terms associated with repeated symbols, which do not occur in permutations.    
\end{rem}

\begin{lemma}[\textbf{Lift identity}]  $S(\sigma) + S(\sigma^+) = 1$.
\end{lemma}
\begin{proof}
    For permutations, the local costs satisfy 
\[
C_{x+1,\,y+1} = -\,C_{xy}, \qquad
C_{\cdot,\,x+1} = \frac{1 - p_x}{2}.
\]
Consequently, $S(\sigma^+) + S(\sigma) = 1.$
\end{proof}

\begin{theorem}\cite[Proposition 10]{engen2020containing}
    For every permutation $\sigma \in \mathbf{S}_k$, either $\sigma$ or $\sigma^+$ occur as an exact subsequence of $zz(k,k+1)$.
\end{theorem}\label{embed-subseq}

\begin{proof}
    Let $R(\sigma)$ be the minimum number of runs needed to place $\sigma$ in a zigzag word. Then $R(\sigma)=k+S(\sigma)$ and $R(\sigma^+)=k+S(\sigma^+).$ From the lift identity, we can conclude that either one of the scores is less than or equal to zero. That implies either $\sigma$ or $\sigma^+$ occur as a subsequence of $zz(k,k+1)$. 
\end{proof}

\subsubsection{Shift identities}
In this section we look at the score of shifting an embedded pattern on the zigzag word  $R_1,R_2 , \cdots$.  Analogous identities for shifting an embedded pattern in $R_2,R_3, \cdots$ is discussed in Appendix \ref{appendix}. 

While verifying if the zigzag word $zz(k,k+1)$ contains a certain pattern, it becomes important to check if the last run of $zz(k,k+1)$ contains part of the pattern. If it doesn't, then we can shorten the length of the superpattern created. Let $S'(\pi)$ be the score of shifting the permutation $\pi$ one run in the zigzag word. Then we have, 
    \begin{equation}\label{shiftsc}
        S'(\pi) = \sum C_{x,y} + \frac{1 - p_{\pi(1)}}{2}. 
    \end{equation}
\begin{lemma} We have the following conditional relations for the shifted score:
\begin{enumerate}
    \item If the score $S(\pi) = 0$ and $\pi(1)$ is even, then $S'(\pi) =-1$. 
 
    \item  If the score $S(\pi) = 0$ and $\pi(1)$ is odd, then $S'(\pi^+)  = 0$.  
    
    \item If the score $S(\pi) = 1$ and $\pi(1)$ is even, then $S'(\pi) = 0$.  

    \item  If the score $S(\pi) = 1$ and $\pi(1)$ is odd, then  $S'(\pi^+)  = -1$.  
\end{enumerate}
The conclusions remain valid if the embedding is shifted by one run to the right or left due to the symmetry.
\end{lemma}\label{shifted_score}
\begin{proof}
  The proof is straightforward from equation (\ref{shiftsc})
    \begin{enumerate}
    \item If the score $S(\pi) = 0$ and $\pi(1)$ is even, then 
    \begin{align*}
        S'(\pi) &= S(\pi) - \left(\frac{1+ p_{\pi(1)}}{2}\right) + \frac{1 - p_{\pi(1)}}{2} \\
        &= -p_{\pi(1)} =-1. 
    \end{align*}
    \item  If the score $S(\pi) = 0$ and $\pi(1)$ is odd, then 
       \begin{align*}
        S'(\pi^+) &= S(\pi^+) - \left(\frac{1+ p_{\pi(1)+1}}{2}\right) + \frac{1 - p_{\pi(1)+1}}{2} \\
        &= 1-p_{\pi(1)+1} = 0.  
    \end{align*}
    
    \item If the score $S(\pi) = 1$ and $\pi(1)$ is even, then 
       \begin{align*}
        S'(\pi) &= S(\pi) - \left(\frac{1+ p_{\pi(1)}}{2}\right) + \frac{1 - p_{\pi(1)}}{2} \\
        &= 1-p_{\pi(1)} = 0.  
    \end{align*}

    \item  If the score $S(\pi) = 1$ and $\pi(1)$ is odd, then 
       \begin{align*}
        S'(\pi^+) &= S(\pi^+) - \left(\frac{1+ p_{\pi(1)+1}}{2}\right) + \frac{1 - p_{\pi(1)+1}}{2} \\
        &= -p_{\pi(1)+1} = -1.  
    \end{align*}
    \end{enumerate}
\end{proof}

\subsection{Circular Score function}
We now extend the concept of the local cost to the circular superpattern setting and prove an analog of Theorem \ref{embed-subseq} in context. 
For a cyclic equivalence class of a permutation, define the local circular cost of placing elements in the zigzag word as in Section~\ref{sec:cost}:
\[
C_{xy} =  \delta_{xy}-\frac{(p_x p_y + 1)}{2} \, \sign(x-y) \, p_x.
\]  
Here, the cost of placing the first element is taken with respect to the last element, i.e.,
\[
C_{\cdot,\sigma(1)} = C_{\sigma(k),\sigma(1)}.
\]
The circular score of the permutation $\sigma$ is then defined by
\[
S^c(\sigma) = \sum_{i=1}^{k} C_{\sigma(i),\, \sigma((i+1) \bmod k)}.
\]
Any cyclic shift of $\sigma$ preserves the multiset of consecutive pairs, and hence $S^c(\sigma)$ is invariant under rotation.

\begin{theorem}
    The quantity $k+S^c$ is the minimum even number of runs required to embed $\sigma$. 
\end{theorem}

\begin{proof}
Since the circular score is invariant under cyclic rotation, we may choose the cyclic representative of $\sigma$ for which the linear score $S$ is minimum. Recall that $k+S$ is the minimum number of runs required for the corresponding linear zigzag embedding. Moreover,
\[
S^c=S-C_{\cdot,\sigma(1)}+C_{\sigma(k),\sigma(1)},
\]
and hence
\[
k+S^c
=
k+S-C_{\cdot,\sigma(1)}+C_{\sigma(k),\sigma(1)}.
\]

We first consider the case when $\sigma(1)$ is odd. If $k+S$ is even, then
\[
-C_{\cdot,\sigma(1)}+C_{\sigma(k),\sigma(1)}=0,
\]
since both $C_{\cdot,\sigma(1)}$ and $C_{\sigma(k),\sigma(1)}$ are zero. If $k+S$ is odd, then
\[
-C_{\cdot,\sigma(1)}+C_{\sigma(k),\sigma(1)}=\pm1,
\]
depending on whether $\sigma(k)<\sigma(1)$ or $\sigma(k)>\sigma(1)$. Thus, $k+S^c$ is even and differs from the minimum linear embedding length by at most one.

Now suppose that $\sigma(1)$ is even. If $k+S$ is even, then $\sigma(k)$ is also even. Hence
\[
-C_{\cdot,\sigma(1)}+C_{\sigma(k),\sigma(1)}=0,
\]
since $C_{\cdot,\sigma(1)}=1$ and $C_{\sigma(k),\sigma(1)}=1$. Finally, if $k+S$ is odd, then $\sigma(k)$ is odd. However, this contradicts the minimality of the chosen linear embedding, since the entries on the last run can be placed on the first run, producing an embedding with fewer runs. Therefore, this case cannot occur.

Consequently, $k+S^c$ is the minimum even number of runs required to embed the circular permutation. Equivalently,
\[
S^c=(\text{minimum even number of runs required to embed }\sigma)-k.
\]
\end{proof}

 \begin{lemma}[\textbf{Circular lift identity}]
     For a permutation $\sigma\in S_k$, we have $S^c(\sigma) + S^c(\sigma^+) = 0$.
 \end{lemma}
 \begin{proof}
 Since $C_{(x+1)\bmod k,(y+1)\bmod k} = -C_{x,y}$, we have $S^c(\sigma^+) = -S^c(\sigma)$ which gives us the desired identity.
\end{proof}
\begin{lemma}\label{oddsum}
For odd $k$ and a permutation $\sigma\in S_k$, $S^c(\sigma)\neq 0$.
\end{lemma}

\begin{proof}
Suppose, for contradiction, that $S^c(\sigma)=0$. Then, by the circular lift identity, we also have $S^c(\sigma^+)=0$.

The score depends only on consecutive pairs $(x_i,x_{i+1})$ (with indices taken modulo $k$) and on their parities $p_{x_i}\in\{\pm1\}$.  Let $\xi_i = p_{x_i}p_{x_{i+1}}, \qquad 1\le i\le k.$ If a consecutive pair has opposite parity, then $\xi_i=-1$ and $C_{x_i x_{i+1}}=0$.  If the pair has the same parity, then $\xi_i=1$ and $C_{x_i x_{i+1}}=\pm1$.

Since $S^c(\sigma)=0$, the positive and negative contributions to the score must cancel. This forces the number of zero-cost transitions (i.e., those with $p_{x_i}p_{x_{(i+1)\bmod k}}=-1)$) to be odd. Consequently,
\[
\prod_{i=1}^k p_{x_i}p_{x_{(i+1)\bmod k}} = -1.
\]
On the other hand, by rearranging terms, we can write
\[
\prod_{i=1}^k \xi_i
   = \prod_{i=1}^k p_{x_i}p_{x_{i+1 \bmod k}}
   = \prod_{i=1}^k p_{x_i}^2
   = 1,
\]
since $p_{x_i}^2=1$ for all $i$. This is a contradiction.
\end{proof}

When $k$ is even, a permutation $\sigma\in S_k$, having an alternating parity sequence $(\text{even}, \text{odd}, \text{even}, \text{odd}, \dots)$, yields $S^c(\sigma) = 0$. 

\begin{theorem}\label{exactsubseq}
For $k$ odd, every permutation $\sigma$ of length $k$, or $\sigma^+$ occurs as an exact subsequence of $\mathrm{zz}(k-1,k+1)$, up to cyclic permutation.
\end{theorem}

\begin{proof}
Using the circular lifting identity and the fact that $S^c(\sigma)\neq 0$ for odd $k$ we have,
\[
S(\sigma)\le -1 \quad \text{or} \quad S(\sigma^+)\le -1.
\]
Therefore, either $\sigma$ or $\sigma^+$ can be placed using at most $k-1$ runs, which implies that it occurs as a subsequence of $\mathrm{zz}(k-1,k+1)$.
\end{proof}
\subsubsection{Even $k$ case:}

When $k$ is even, consider $zz(k-1,k+1)$. The $C_{\cdot,x}$ argument no longer holds because the zigzag pattern is not symmetric (that is, the number of runs is odd). Therefore in this case we prove a weaker theorem than Theorem \ref{exactsubseq}.\\

\begin{theorem}
     For $k$ even, every pattern of length $k$ occurs as an order isomorphic circular subsequence of $zz(k-1,k+1)$. 
\end{theorem}
\begin{proof}
Given any $\sigma \in \mathbf{S}_k$, consider its rotation $\rho_r(\sigma)$ such that the last element is $k$. Identify $k$ with the last element of $zz(k-1,k+1)$, that is $k+1$. \par
Let $\tau = \rho_r(\sigma)|_{1:k-1}$ (the first $k-1$ elements of rotated permutation). Then $\tau \in \mathbf{S}_{k-1}$. From the Theorem \ref{exactsubseq}, either $\tau$ or $\tau^+$ occurs as an exact subsequence of $zz(k-1,k)$ up to cyclic permutation. But, $zz(k-1,k)$ is a subsequence of $zz(k-1,k+1)$. Therefore any $\sigma$ has a rotation such that it occurs as a order isomorphic subsequence of $zz(k-1,k+1)$.
\end{proof}

\section{A construction for odd circular superpattern}

Having established the score function, we now construct a circular superpattern. The proof follows the philosophy of Engen and Vatter but requires new arguments to handle the modified terminal portion of the zigzag word.

In this section, we give a construction to create an odd circular superpattern. We first recall the method (called \emph{breaking ties}) given by Engen and Vatter \cite{engen2020containing} to obtain a permutation from the corresponding zigzag word (This procedure is similar to level order traversal on trees, here its instead on the zigzag words).

For ease of notation, let $\omega = zz(m,n)$. The procedure of breaking ties is as follows:


\begin{itemize}
    \item We first view the zigzag word~$\omega$ in the $(X,Y)$-plane, as illustrated in Figures~\ref{zz(4,4)} and~\ref{zz(5,5)}. 
    In this representation, each element $\omega_i$ corresponds to the point $(i, \omega_i)$.
    
    \item We refer to the $Y$-coordinates as \emph{levels}. 
    Starting from the rightmost element in the lowest level, we traverse each level from right to left. 
    If the first element encountered in this traversal is $(s, \omega_s)$, we assign $\zeta_s = 1$.
    
    \item More generally, if the $j$-th element visited in this traversal is $(g, \omega_g)$, then we assign $\zeta_g = j$.
    
    \item Once all elements at one level are processed, we proceed to the next higher level and repeat the traversal until all levels have been exhausted.
\end{itemize}

This procedure yields the desired permutation~$\zeta$. It is illustrated for $zz(3,3)$ in Figure \ref{wtoperm}. 

\begin{figure}[h!]

\begin{center}
\begin{tikzpicture}
    \def\seq{1,3,2,1,3}
    \foreach [count=\i] \y in \seq {
        \node at (\i,\y) {\y};
    }
        \foreach [count=\i] \val in \seq {
        \node[below=8pt] at (\i,0) {\val};
    }

    \draw[-, dashed] (0.5,0) -- (5.5,0);
      \node[left=6pt] at (0.5,-0.6) {$\omega$};
\end{tikzpicture}
\hspace{1cm}
\raisebox{6em}{$\Longrightarrow$} 
\hspace{1cm}
\begin{tikzpicture}
    \def\seq{1,3,2,1,3}
    \def\seqtwo{2,5,3,1,4}
    \newcommand{\getfromseqtwo}[1]{%
        \pgfmathparse{{2,5,3,1,4}[#1-1]}%
        \pgfmathresult%
    }
    \foreach [count=\i] \y in \seq {
        \node at (\i,\y) {\getfromseqtwo{\i}};
    }
    \foreach [count=\i] \val in \seqtwo {
        \node[below=8pt] at (\i,0) {\val};
    }

    \draw[-, dashed] (0.5,0) -- (5.5,0);
     \node[left=6pt] at (0.5,-0.6) {$\zeta$};
\end{tikzpicture}
\end{center}

    \caption{$zz(3,3)$ and corresponding $\zeta$}
    \label{wtoperm}
\end{figure}
Engen and Vatter proved that for odd $n$, $\zeta$ is a superpattern. For even $n$, the permutation $\zeta$ obtained from $zz(n,n)$ appended with 1, is a $k$-superpattern. 

\subsection{An odd circular superpattern}\label{odd}
We will now use the method of \emph{breaking ties} to construct a circular superpattern. Let $k>3$ be an odd number and consider the zigzag word $\omega = (zz(k-1,k-1),1)$. The last two elements of $\omega$ are $(2,1)$. Let $\mathrm{zzc}(k) = (zz(k-1,k-1)|_{1:L-2}, k,k-1)$ by replacing the terminal $(2,1)$ of $\omega$ by $(k,k-1)$.  Let $\gamma_k$ be the permutation obtained by $\mathrm{zzc}(k)$ following the procedure of \emph{breaking ties}. 
\begin{eg}
    When $k = 5$, we have \[
\omega = 1\,3\,4\,2\,1\,3\,4\,2\,1,
\qquad\text{ and }\qquad
\mathrm{zzc}(5) = 1\,3\,4\,2\,1\,3\,4\,5\,4.
\]
This results into the permutation, $\gamma_5 = (2,5,8,3,1,4,7,9,6)$. It is experimentally verified that $\gamma_5$ is a 5-circular superpattern. This is illustrated in Figure \ref{tiebreakcirc}  
\end{eg}

\begin{figure}[h!]
\centering

\begin{tabular}{c}
\begin{tikzpicture}
    \def\seq{1,3,4,2,1,3,4,5,4}
    \foreach[count=\i] \y in \seq{
        \node at (\i,\y) {\y};
    }
    \foreach[count=\i] \val in \seq{
        \node[below=8pt] at (\i,0) {\val};
    }
    \draw[dashed] (0.5,0)--(9.5,0);
    \node[left] at (0.5,-0.5) {$\text{zzc}(5)$};
\end{tikzpicture}
\\[1em]
$\Downarrow$

{\small Breaking ties}
\\[1em]

\begin{tikzpicture}
    \def\seq{1,3,4,2,1,3,4,5,4}
    \def\perm{2,5,8,3,1,4,7,9,6}

    \foreach[count=\i] \y in \seq{
        \pgfmathparse{{2,5,8,3,1,4,7,9,6}[\i-1]}
        \node at (\i,\y) {\pgfmathresult};
    }

    \foreach[count=\i] \v in \perm{
        \node[below=8pt] at (\i,0) {\v};
    }

    \draw[dashed] (0.5,0)--(9.5,0);
    \node[left] at (0.5,-0.5) {$\zeta$};
\end{tikzpicture}
\end{tabular}

\caption{$\text{zzc}(5)$ and the corresponding circular superpattern candidate permutation $\zeta$ obtained by breaking ties.}
\label{tiebreakcirc}
\end{figure}

 Before proving the main result, we introduce the notions of \emph{distant inverse-descent} and \emph{layered permutations}, which play a central role in the proof.

\begin{definition}\cite[Distant inverse-descent]{engen2020containing}
    We say that two entries $\pi(j)$ and $\pi(k)$ of the permutation $\pi$ form an \emph{inverse-descent} if $j<k$ and $\pi(j)=\pi(k)+1$. If $\pi(j)$ and $\pi(k)$ form an inverse-descent and are not adjacent in $\pi$ (that is $k\geq j+2$), then they form a \emph{distant inverse-descent.}
\end{definition}
\begin{eg}
    The pair $\pi(3)$ and $\pi(5)$ form a distant inverse-descent in $\pi=(264513).$
\end{eg}

\begin{definition}\cite[Layered permutation]{engen2020containing}
    The sum of two permutations, $\pi$ of length $m$ and $\sigma$ of length $n$ is the permutation of length $m+n$ defined as, $$(\pi\oplus\sigma)(i)=\begin{cases}
        \pi(i) &\text{if }1\leq i\leq m\\
        \sigma(i-m)+m&\text{if }m+1\leq i\leq m+n.
    \end{cases}$$
    A permutation is said to be \emph{layered} if it can be expressed as a sum of decreasing permutations. The decreasing shifted components are called the layers.
\end{definition}
\begin{eg}
    The permutation $(21365487)$ is a layered permutation with layers $21, 3, 654,87$. The permutation $(264513)$ is not layered.
\end{eg}
It is known \cite[Proposition 12]{engen2020containing} that a permutation is layered if and only if it does not have a distant inverse-descent.

Now we state the main theorem of the current section,

\begin{theorem}\label{zzc_odd}
Let $k>3$ be odd and let $\gamma_k$ be the permutation obtained from the word $zzc(k)$ by the procedure of breaking ties, described above. 
Then $\gamma_k$ is a circular $k$-superpattern.
\end{theorem}

In order to prove the theorem, first we establish several technical Lemmas.

\subsection{Embedding framework and terminology}

Before proving the lemmas, we first clarify the objects involved and the terminology used throughout the remainder of this section. Figure~\ref{fig:embedding-philosophy} summarizes the overall workflow.

Our objective is to show that the permutation $\gamma_k$, obtained from the zigzag word $zzc(k)$ by the breaking-ties procedure, is a circular $k$-superpattern. Therefore, given any permutation $\sigma\in S_k$, we must identify a cyclic rotation of $\sigma$ that occurs as an order-isomorphic subsequence of $\gamma_k$.

Since the breaking-ties procedure is fixed, it is sufficient to work with the underlying zigzag word $zzc(k)$. Thus, instead of directly constructing an occurrence in $\gamma_k$, we identify a suitable cyclic rotation of $\sigma$ and embed it into $zzc(k)$ in a manner consistent with the breaking-ties procedure. The corresponding order-isomorphic subsequence of $\gamma_k$ then follows automatically.

Consequently, each of the subsequent lemmas consists of two conceptual steps:
\begin{enumerate}
    \item identifying a suitable cyclic rotation of the given permutation, and
    \item constructing an embedding of that rotation into the underlying zigzag word $zzc(k)$.
\end{enumerate}

Once these two steps are completed, the desired occurrence in $\gamma_k$ is immediate.

Therefore, throughout the proofs of Lemmas~5.7--5.10, our objective is only to construct suitable embeddings into $zzc(k)$. The passage from $zzc(k)$ to $\gamma_k$ is automatic and will not be mentioned further.

\paragraph{Embedding.}
We identify the zigzag word $zzc(k)$ as consisting of the runs
$R_1,R_2,\ldots,R_{k-1}$,
where the final decreasing run $R_{k-1}$ consists of
$k-1,k-3,\ldots,4$.
This is followed by the terminal list
$L=(k,k-1)$. Thus $zzc(k) = R_1,R_2, \cdots R_{k-1},L$.

For brevity, we denote by $(a,R_i)$ the occurrence of the symbol $a$ in the
$i$-th run of $zzc(k)$. An assignment
$x\mapsto(a,R_i)$
means that the symbol $x$ of the pattern is mapped to the occurrence $(a,R_i)$
of the underlying zigzag word. Similarly,
$x\mapsto(a,L)$
means that the symbol $x$ is mapped to the occurrence of $a$ in the terminal list $L$.

Once every symbol of the pattern has been assigned to an occurrence of $zzc(k)$, the resulting assignment determines an embedding.

In the sequence of results from Lemma~\ref{layered2} to the proof of Theorem~\ref{zzc_odd}, we shall assume that $k>3$ is odd, $\sigma\in S_k$ and $\mathrm{zzc}(k)$ and $\gamma_k$ are as described at the beginning of Section~\ref{odd}.

\begin{figure}[h]
\centering
\begin{tikzpicture}[
    node distance=1.6cm,
    >=Stealth,
    every node/.style={align=center},
    box/.style={draw, rounded corners, minimum width=4cm, minimum height=0.9cm}
]

\node[box] (sigma) {Pattern $\sigma$};

\node[box, below of=sigma] (rotation)
{Choose a cyclic\\representative};

\node[box, below of=rotation] (embed)
{Embed into the underlying\\zigzag word $zzc(k)$};

\node[box, below of=embed] (ties)
{Breaking ties};

\node[box, below of=ties] (gamma)
{Circular occurrence in\\$\gamma_k$};

\draw[->] (sigma) -- (rotation);
\draw[->] (rotation) -- (embed);
\draw[->] (embed) -- (ties);
\draw[->] (ties) -- (gamma);

\end{tikzpicture}
\caption{Embedding framework used throughout Section~5. Every proof first identifies a suitable cyclic representative of the pattern as a subsequence of the underlying zigzag word $zzc(k)$. The corresponding occurrence in $\gamma_k$ is then obtained automatically by the fixed breaking-ties procedure.}
\label{fig:embedding-philosophy}
\end{figure}

\begin{lemma}\label{layered2}
    If $\sigma(k-1) = 2$ and $\sigma$ is layered, then $\gamma_k$ circularly contains $\sigma$.
\end{lemma}
\begin{proof}
If $\sigma(k-1) = 2$ and $\sigma$ is layered then the only such permutations are the decreasing permutation $(k,k-1,\cdots,2,1)$ and $(1,k-1,k-2,\cdots, 2,k)$.\par
The permutation $(k-2,k-3,\cdots,1,k,k-1)$ is cyclically equivalent to $(k,k-1,\cdots,2,1)$. We embed $(k-2,k-3,\cdots,1,k,k-1)$ as an exact subsequence of $\mathrm{zzc}(k)$ by the following assignment: 
\begin{align*}
    k  &\mapsto (k,L) \\
    k-1 &\mapsto (k-1,L) \\
    k-2 &\mapsto (k-2,R_1) \\
    k-3 &\mapsto (k-3,R_2) \\
    &\vdots \\
    1 &\mapsto (1,R_{k-2}).
\end{align*}
\par
Next consider, $\sigma=(1,k-1,k-2,\cdots, 2,k)$ which is cyclically equivalent to $(k-2,k-3,\cdots,2,k,1,k-1)$. We embed $(k-2,k-3,\cdots,2,k,1,k-1)$ as an order isomorphic subsequence of $\mathrm{zzc}(k)$ by the following assignment: 
\begin{align*}
    1 &\mapsto (1,R_{k-2}) \\
    k-1 &\mapsto (k-1,R_{k-1}) \\
    k-2 &\mapsto (k-2,R_1) \\
    k-3 &\mapsto (k-3,R_2) \\
    &\vdots \\
    4 &\mapsto (4, R_{k-5}) \\
    2 &\mapsto (1 ,R_{k-4}) \\
    k &\mapsto (k-1,R_{k-3}).
\end{align*}

\end{proof}

\begin{lemma}\label{layerd3}
    If $\sigma(k-1) = 3$ and $\sigma$ is layered, then $\gamma_k$ circularly contains $\sigma$.
\end{lemma}
\begin{proof}
    If $\sigma(k-1) = 3$ and $\sigma$ is layered then there are two such permutations, $(2,1,k-1,k-2,\cdots,3,k)$ and $(1,2,k-1,k-2,\cdots,3,k)$. \par
    If $\sigma=(2,1,k-1,\cdots,3,k)$, then we construct an order-isomorphic embedding by the following assignment:
\begin{align*}
3      &\mapsto (3,R_1),\\
k      &\mapsto (k-1,R_2),\\
2      &\mapsto (2,R_2),\\
1      &\mapsto (1,R_3),\\
k-1    &\mapsto (k-1,R_4),\\
k-2    &\mapsto (k-2,R_5),\\
k-3    &\mapsto (k-3,R_6),\\
&\vdots\\
4      &\mapsto (4,R_{k-1}).
\end{align*}
This assignment defines an order-isomorphic embedding of the chosen cyclic representative into $zzc(k)$. \\

If $\sigma=(1,2,k-1,k-2,k-3,\cdots,3,k)$, then we construct an order-isomorphic embedding by the following assignment:
\begin{align*}
3      &\mapsto (3,R_1),\\
k      &\mapsto (k-1,R_2),\\
1      &\mapsto (1,R_3),\\
2      &\mapsto (3,R_3),\\
k-1    &\mapsto (k-1,R_4),\\
k-2    &\mapsto (k-2,R_5),\\
&\vdots\\
4      &\mapsto (4,R_{k-1}).
\end{align*}
This assignment defines an order-isomorphic embedding into $zzc(k)$.
\end{proof}

We now recall the construction of a word from a distance inverse descent permutations that was given by Engen Vatter \cite{engen2020containing}. We term this as the horizontal cut procedure.

\begin{definition}[Horizontal cut embedding]\label{lem:horizontal-cut}
Let $\sigma$ be a permutation of length $k$ that has a distant inverse
descent, i.e.\ there exist indices $a<b-1$ such that
$\sigma(a)=\sigma(b)+1$.  Define the word $p\in [k-1]^k$ by
\[
p(i)=
\begin{cases}
\sigma(i), & \sigma(i)\le \sigma(b),\\
\sigma(i)-1, & \sigma(i)\ge \sigma(a).
\end{cases}
\]
Then $p$ contains exactly two occurrences of the letter $\sigma(b)$
and these occurrences are not adjacent.  In particular $p$ has no
immediate repetition.

Consequently, by \cite[Proposition 10]{engen2020containing}, either $p$ or $p^{+}$ occurs as a subsequence of the first $k$ runs of the infinite zigzag word.
\end{definition}

Before proving the next lemma, we explain why only one exceptional case requires additional analysis.

The horizontal cut construction of Engen and Vatter \cite{engen2020containing} reduces every permutation with a distant inverse-descent to a word that can be embedded into the ordinary zigzag word. In our construction, however, the only modification to the zigzag word is the replacement of the terminal $(2,1)$ by $(k,k-1)$. Consequently, every embedding that avoids the terminal $(2,1)$ remains valid without modification.

Furthermore, in the horizontal cut construction, if the lifted word $(\sigma')^+$ is embedded, then the terminal symbol involved is $3$, so the modified terminal portion of $zzc(k)$ is never used. Therefore, the only genuinely new situation occurs when the horizontal cut construction embeds $\sigma'$ itself and the original permutation satisfies $\sigma(k-1)=2$. The following lemma resolves precisely this remaining boundary case.

\begin{lemma}\label{h_cut_max}
    If the horizontal cut is performed on a permutation $\pi'$ to obtain $\pi$ and the last element of the permutation $\pi$ of length $k-1$ and $\pi(k-1) =2$ utilizing $R_{k-1}$ in the embedding, then the maximum element of $\pi$ is $k-2$ with multiplicity one and second maximum is $k-3$.  
\end{lemma}
\begin{proof}
    For contradiction assume $k-2$ occurs twice, that is the only distance inverse occurs with respect to $k,k-1$. Then such permutation will always have the form $k , 1,2,3,4, \cdots m , k-1,m+1, \cdots k-2$. Then it cannot have 2 in the last run. Thus giving the contradiction.  
\end{proof}

\begin{lemma}\label{2boundaryDID}
If $\sigma(k-1)=2$ and $\sigma$ has distance inverse descent, then $\gamma_k$ circularly contains $\sigma$.
\end{lemma}
\begin{proof}
Let $\sigma \in S_k$ satisfy $\sigma(k-1)=2$ and assume that $\sigma$
has a distant inverse descent. 
Let $\sigma'$ be the word after the horizontal cut operation on $\sigma(1:(k-1))$. Then there are two cases, either $(\sigma')^+$ is an exact subsequence of $zz(k-1,k-1)$ or $\sigma'$ is an exact subsequence of $zz(k-1,k-1)$. \\
Case I :
If $(\sigma')^+$ is an exact subsequence of $zz(k-1,k-1)$ then $((\sigma')^+,k)$ is embedded in $zzc(k)$ (by assignment $k \mapsto (k,L)$). \\

Case II : When $\sigma'$ is an exact subsequence of $zz(k-1,k-1)$. The three scenarios are depicted below:\\
Case II(a)   $S(\pi) = -1$:
Now suppose that $\sigma'$ is an exact subsequence of $zz(k-1,k-1)$, and let
$\pi=\sigma'(1:k-2)$.
If $S(\pi)=-1$, then the exact embedding of $\pi$ occupies only the first
$k-3$ runs. Since the symbol $2$ does not occur in $\pi$, the occurrence
$(2,R_{k-3})$ remains unused. Hence we can extend the embedding by the
assignments
$
2\mapsto(2,R_{k-3}), 
k\mapsto(k,L).
$
Therefore, $\sigma$ appears as an order-isomorphic subsequence of $\gamma_k$.
Now when $S(\pi) \in \{ 0, 1\}$ we take the help of Lemma \ref{shifted_score} to embed $\pi$ or $\pi^+$ in $zzc(k)$. \\

Case II(b) $S(\pi) = 0$:
From Lemma \ref{shifted_score}, if $\pi(1)$ is even, then shifting the assignment will give us $S'(\pi) = -1$. Then we can have the assignment $2 \mapsto (1,R_1)$ and $k \mapsto (k-1,R_2)$. The corner case here is when $(k-1,R_2)$ is not free for an assignment, but notice that this case will never happen since $\pi$ is obtained by a horizontal cut procedure. \\

From Lemma \ref{shifted_score}, if $\pi(1)$ is odd, then shifting the assignment will give us $S'(\pi^+) = 0$. Then we can have the assignment $2 \mapsto (1,R_1)$ and $k \mapsto (k-1,R_2)$. 

The corner case here is when $\pi(1) = k-2$. We make the assignment $k-2 \mapsto (k-1,L)$. Note that largest element of $\pi$ is $k-2$ with multiplicity one and the second largest is $k-3$ from Lemma \ref{h_cut_max}. In this case $S_e(\pi(2:k-2))$ is zero (Lemma \ref{even_shifted_score}). Then we can left shift $\pi(2:k-2)$ or $\pi^+(2:k-2)$, to have the score 0 or -1, thus keeping the $R_{k-2}$ free. Then we can have the assignment $2 \mapsto (3,R_{k-3}), k \mapsto (k,L)$. \\
Case II(c) $S(\pi) = 1$, 
 we may shift either $\pi$ or
$\pi^{+}$ by one run in the zigzag word. This shift creates a position
in the first run thus we can have $2 \mapsto (1,R_1),  k \mapsto (k-1,R_2) $. \\
The boundary conditions are precisely those permutations for which the $k-1$ of the second run is not free after the shift. That is $\pi(1) = k-2$ and we are shifting $\pi^+$. 
If for a boundary case, $S(\pi)=0$, then we assign $\pi(1) \mapsto (k-1,L)$. We left shift the $\pi(2:k-2)$ or $\pi^+(2:k-2)$ such that it is accommodated in the first $k-3$ runs. Then $2 \mapsto (3,R_{k-2})$. 

The only remaining potential obstruction would occur if $S(\pi)=1$ and
$\pi(k-2)=1$. However, this configuration would force the final run of
the zigzag word to contain the entries $1$ and $2$ in increasing order,
which is impossible by the construction of the runs. Hence this case
cannot occur.

Therefore a cyclic rotation of $\sigma$ appears as an order-isomorphic
subsequence of $\gamma_k$, completing the proof.
\end{proof}

\subsubsection*{ Proof of Theorem \ref{zzc_odd}:}
With these lemmas, we can now prove the Theorem \ref{zzc_odd}.
\begin{proof}
Let $\sigma \in S_k$ be arbitrary. Since circular containment is invariant 
under cyclic rotation, we may assume without loss of generality that 
$\sigma(k)=k$.

We use the structural classification of Engen and Vatter~\cite{engen2020containing}, 
which states that every permutation is either layered or has a distant 
inverse descent.

\medskip
\noindent
\textbf{Case 1: $\sigma$ has a distant inverse descent.}

In this case the embedding of the prefix $\sigma(1:k-1)$ into the zigzag 
word follows from the Engen--Vatter construction (when $\sigma(k-1) \neq 2$). 
In particular, the embedding uses only the first $k-1$ runs of the zigzag 
word and therefore avoids the modified terminal portion of $zzc(k)$. 
Consequently $\sigma(1:k-1)$ appears as a subsequence of $zzc(k)$, and 
the final element $k$ can be assigned to the distinguished terminal entry 
of $zzc(k)$. Hence $\sigma$ appears as a circular subsequence of $\gamma_k$. The boundary case, when $\sigma(k-1) = 2$ is handled by Lemma \ref{2boundaryDID}.

\medskip
\noindent
\textbf{Case 2: $\sigma$ is layered.}

From Engen-Vatter, when $\sigma(k-1) \notin \{ 1,2,3 \}$ we can embed $\sigma(1:k-1)$ in the $k-1$ runs of 
$zzc(k)$. 
The boundary cases $\sigma(k-1) \in \{1,2 \}$ and $\sigma(k-1)=3$ are handled 
in Lemmas~\ref{layered2} and~\ref{layerd3}, respectively, 
which show that suitable rotations of these permutations occur in 
$\gamma_k$.

\medskip
In all cases we obtain a cyclic representative of $\sigma$ as a subsequence 
of $\gamma_k$. Therefore $\gamma_k$ contains every permutation of length 
$k$ up to cyclic equivalence, and hence $\gamma_k$ is a circular 
$k$-superpattern.
\end{proof}

\section{Conclusion and future scope}
We have formalized circular $k$-superpatterns and established the general upper bound $L_{\mathrm{circ}}(k)\le L(k-1)+1$, together with parity-based zigzag constructions for both odd and even $k$.  

Open questions include: 
\begin{itemize}
\item Determining tight bounds or asymptotic growth of $L_{\mathrm{circ}}(k)$. 

\item Designing algorithms to generate minimal circular superpatterns. 

\item Exploring analogues of circular superpatterns in cyclic networks and biological sequences. 

\item Pattern avoidance problems in circular settings.

\end{itemize}

\section{Appendix}\label{appendix}
Consider the runs $R_2,R_3,R_4 \cdots $ on a zigzag word, let us call it an even zigzag word, then score of placing $\pi$ on these runs is given by
\begin{equation}\label{ssigma}
S_e(\sigma) = \sum_{i=1}^{k-1} C_{\sigma_i \sigma_{i+1}}
         + C^e_{\cdot,\,\sigma_1},    
\end{equation}
where the initial term is
\[
C^e_{\cdot,\,x} = \frac{1 - p_x}{2}.
\]

We can also verify that the Lift identity holds in this case, i.e. $S_e(\sigma) + S_e(\sigma^+) = 1$.
\begin{lemma} We have the following conditional relations for the shifted score:
\begin{enumerate}
    \item If the score $S_e(\pi) = 0$ and $\pi(1)$ is even, then $S_e'(\pi^+) = 0$. 
 
    \item  If the score $S_e(\pi) = 0$ and $\pi(1)$ is odd, then $S_e'(\pi^+)  = -1$.  
    
    \item If the score $S_e(\pi) = 1$ and $\pi(1)$ is odd, then $S_e'(\pi) = 0$.  

    \item  If the score $S_e(\pi) = 1$ and $\pi(1)$ is even, then  $S'(\pi^+)  = -1$.  
\end{enumerate}
The conclusions remain valid if the embedding is shifted by one run to the right or left due to the symmetry.
\end{lemma}\label{even_shifted_score}

\begin{proof}
  The proof is straightforward from equation (\ref{shiftsc})
    \begin{enumerate}
    \item If the score $S_e(\pi) = 0$ and $\pi(1)$ is odd, then 
    \begin{align*}
        S_e'(\pi) &= S_e(\pi) - \left(\frac{1 - p_{\pi(1)}}{2}\right) + \frac{1 + p_{\pi(1)+1}}{2} \\
        &=  p_{\pi(1)} = -1. 
    \end{align*}

    \item  If the score $S_e(\pi) = 0$ and $\pi(1)$ is even, then 
    \begin{align*}
        S_e'(\pi^+) &= S_e(\pi^+) - \left(\frac{1 - p_{\pi(1)+1}}{2}\right) + \frac{1 + p_{\pi(1)+1}}{2} \\
        &= 1 + p_{\pi(1)+1} = 0. 
    \end{align*}
    
    \item If the score $S_e(\pi) = 1$ and $\pi(1)$ is even, then 
       \begin{align*}
        S_e'(\pi^+) &= S_e(\pi^+) - \left(\frac{1- p_{\pi(1)+1}}{2}\right) + \frac{1 + p_{\pi(1)+1}}{2} \\
        &= 0+p_{\pi(1)+1} = -1.  
    \end{align*}

    \item  If the score $S_e(\pi) = 1$ and $\pi(1)$ is odd, then 
       \begin{align*}
        S_e'(\pi) &= S_e(\pi) - \left(\frac{1+ p_{\pi(1)}}{2}\right) + \frac{1 - p_{\pi(1)}}{2} \\
        &= 1+ p_{\pi(1)} = 0.  
    \end{align*}
    \end{enumerate}
\end{proof}

\end{document}